\documentclass[12pt]{amsart}
\usepackage{amssymb,amsmath,amscd,graphicx,
latexsym,amsthm}
\usepackage{amssymb,latexsym,amsmath,amscd,graphicx,array}
\usepackage[mathscr]{eucal}
  \usepackage[all]{xy}
\usepackage{hyperref}

\usepackage[all]{xy} \renewcommand{\P}{ \ensuremath{\mathbb{P}}}
\usepackage{layout}

\title{Convex bodies appearing as Okounkov bodies of divisors}

\theoremstyle{plain}
\newtheorem*{theoremA}{Theorem A}
\newtheorem*{theoremB}{Theorem B}
\newtheorem{theorem}{Theorem}[section]

\newtheorem{proposition}[theorem]{Proposition}
\newtheorem{corollary}[theorem]{Corollary}

\theoremstyle{definition}

\theoremstyle{remark} 
\newtheorem{remark}[theorem]{Remark} 
\newtheorem{example}[theorem]{Example}

\def\Z{{\mathbb Z}} 
\def\N{{\mathbb N}}

\def\R{{\mathbb R}} 
\def\Q{{\mathbb Q}} 
\def\P{{\mathbb P}} 
\def\OO{{\mathcal O}}

\newcommand{\deq}{\ensuremath{ \stackrel{\textrm{def}}{=}}} 
\newcommand{\equ}{\ensuremath{ \,=\, }}

\def\rat{\dashrightarrow}

\DeclareMathOperator{\Supp}{Supp}

\begin{document}
\author{Alex K\"uronya}
\address{Budapest University of Technology and Economics, Department of Algebra, Budapest P.O. Box 91, H-1521 Hungary}
\email{{\tt alex.kuronya@math.bme.hu}}

\author{Victor Lozovanu}
\address{University of Michigan, Department of Mathematics, Ann Arbor, MI 48109-1109, USA}
\email{{\tt vicloz@umich.edu}}

\author{Catriona Maclean}
\address{Universit\'e Joseph Fourier, UFR de Math\'ematiques, 100 rue des Maths, BP 74, 38402 St Martin d'H\'eres, France}
\email{\tt Catriona.Maclean@ujf-grenoble.fr}

\begin{abstract}
Based on the work of Okounkov (\cite{Ok96}, \cite{Ok03}), Lazarsfeld and 
Musta\c t\u a (\cite{LM08}) and Kaveh and Khovanskii (\cite{KK08}) have
independently associated a convex body, called the Okounkov 
body, to a big divisor on a smooth projective variety with respect to a 
complete flag. In this paper we consider the following question: what can be 
said about the set of convex bodies that appear as Okounkov bodies? We show 
first that the set of convex bodies appearing as Okounkov bodies of 
big line bundles on smooth projective varieties with respect to  
admissible flags is countable. We then give a complete characterisation of the
set of convex bodies that arise as Okounkov bodies of $\R$-divisors on 
smooth projective surfaces. Such Okounkov bodies are always polygons, 
satisfying certain combinatorial criteria. Finally, we construct two examples of
non-polyhedral Okounkov bodies. In the first one, the variety we deal with is Fano 
and the line bundle is ample. In the second one, we find a Mori dream space variety
such that under small perturbations of the flag the Okounkov body remains non-polyhedral.
\end{abstract}

\thanks{During this project the first author was partially supported by  the CNRS, the  DFG-Leibniz program, the SFB/TR 45 ``Periods, moduli spaces and arithmetic of algebraic varieties'',  the OTKA Grants 61116, 77476, 77604, and a Bolyai Fellowship by the Hungarian Academy of Sciences. }

\maketitle

\section*{Introduction}

Let $X$ be a smooth complex projective variety of dimension $n$ and let  
$\mathcal{L}$ be a big line bundle on $X$. Suppose given a flag
\[
Y_{\bullet} \ : \ X=Y_0 \ \supseteq \ Y_1 \ \supseteq \ Y_2 \ \supseteq \ 
\ldots \ \supseteq \ Y_{n-1} \ \supseteq 
\ Y_n=\{ \textup{pt} \}
\] 
of irreducible and smooth subvarieties on $X$ with $\textup{codim}_X(Y_i)=i$. 
We call this an admissible flag. In \cite{LM08}, Lazarsfeld and 
Musta\c t\u a, inspired by the work of Okounkov (\cite{Ok96}, \cite{Ok03}), 
construct a convex body $\Delta _{Y_{\bullet}}(X;\mathcal{L})$ in $\R^n$, 
called the Okounkov body, associated to $\mathcal{L}$ and $Y_{\bullet}$. 
This body  encodes the asymptotic behaviour of the linear series 
$|\mathcal{L}^{\otimes n}|$. Lazarsfeld and  Musta\c t\u a link its properties to the 
geometry of $\mathcal{L}$. For example, because here $\mathcal{L}$ is big we have 
that
\[
\textup{vol}_X(\mathcal{L}) \ = \ n! \cdot \textup{vol}_{\R^n} 
(\Delta_{Y_{\bullet}}(X;\mathcal{L}))
\]
where the right hand side is the Euclidean volume of $\Delta_{Y_{\bullet}}
(X;\mathcal{L})$. This enabled Lazarsfeld and Musta\c t\u a to 
simplify the proofs of many basic properties of volumes of line bundles.

We recall the construction of $\Delta_{Y_\bullet}(X;\mathcal{L})$. 
To any effective divisor $D$ on $X$ we associate an integral vector 
\[
\nu_{Y_{\bullet}} (D) \ = \ (\nu_1(D), \ldots ,\nu_n(D)) \in \N^n
\]
defined as follows. We recursively construct numbers $\nu_{i}(D)$ and divisors 
$D_i$ on $Y_i$ in the following manner: 
\begin{enumerate}
\item $D_0= D$, 
\item $\nu_i(D)$ is  the coefficient of $Y_{i}$ in $D_{i-1}$, 
\item $D_{i}= (D_{i-1}- \nu_iY_i)|_{Y_i}$. 
\end{enumerate}
We now set
\[
\Gamma_{Y_{\bullet}}(X;\mathcal{L})_m \ = \ \Big\{ \nu_{Y_{\bullet}}(D) \ | 
\ D=\mbox{zero}(s)\mbox{ for some }0\neq s\in H^0(X,\mathcal{L}^{\otimes m})\Big\}
\subseteq \N^n
\]
 for any $m\in\N$. The Okounkov body $\Delta_{Y_{\bullet}}(X;\mathcal{L})$ is 
then given by 
\[
\Delta_{Y_{\bullet}}(X;\mathcal{L}) \ = \ \text{closed convex hull} 
\Big( \bigcup_{m\geq 1} \frac{1}{m}\Gamma_{Y_{\bullet}}(X;\mathcal{L})_m\Big) 
\subseteq \R^n \ .
\]
If $D$ is a (possibly non-rational) Cartier divisor on $X$ then we define 
$\Delta_{Y_{\bullet}}(X; D)$ as follows:
\[
\Delta_{Y_{\bullet}}(X; D)=  \overline{\{\nu(D') \ | \ D'\geq 0, D'\sim_{\R} D\}},
\]
which is simply $\Delta_{Y_\bullet}(X;\mathcal{O}_X(D))$ when
$D$ is an integral divisor.

In this paper we study the set of convex  bodies appearing as Okounkov bodies of line bundles on smooth projective varieties with 
respect to some admissible flag. Our first result, proved in Section 1, 
shows that this set is countable. 

\begin{theoremA}\label{theoremA}
The collection of all Okounkov bodies is countable. That is, for any natural number $n\geq 1$, there exists a countable set of bounded convex bodies $(\Delta_i)_{i\in\N}\subset \R^n$ such that for any complex smooth projective variety $X$ of dimension $n$, any big
line bundle $\mathcal{L}$ on $X$ and any admissible flag $Y_{\bullet}$ on $X$, the body $\Delta_{Y_{\bullet}}(X;\mathcal{L}) = \Delta_i$ for some $i\in\N$.
\end{theoremA}

The proof of Theorem A is similar to the proof of the countability of
volume functions given in \cite{KLM09}. It was established in \cite{LM08} that
for a variety $X$ equipped with a flag $Y_\bullet$ the Okounkov bodies of 
big real classes on $X$ with respect to $Y_{\bullet}$ fit together in a convex
cone, called the global Okounkov cone. We prove Theorem A by  
analysing the variation of global Okounkov cones in flat families. 

The question then naturally arises whether this countable set of convex bodies
can be characterised. We give an affirmative answer for surfaces.
An explicit description of $\Delta(D)$ for any real divisor $D$ on a smooth surface
$S$ with respect to a flag $(C,x)$ based on the Zariski decomposition 
is given in \cite[Theorem 6.4]{LM08}. It was noted that it followed from this description that the Okounkov body was a  possibly infinite polygon. We give a complete characterisation of Okounkov 
bodies on surfaces based on this work: these turn out to be finite polygons 
satisfying a few extra combinatorial conditions.

\begin{theoremB}\label{theoremB}
The Okounkov body of an $\R$-divisor on a smooth projective surface with respect
to some flag is a finite polygon. 
Up to translation, a real polygon $\Delta\subseteq\R^{2}_+$ is the Okounkov 
body of an $\R$-divisor $D$ on a smooth projective surface $S$ 
with respect to a complete flag $(C,x)$ if and only if 
\[ 
\Delta = \{ (t,y)\in \mathbb{R}^2\ | \ \nu\leq t\leq \mu, \alpha(t)\leq y\leq 
\beta(t)\}\ 
\] 
for certain real numbers  $0\leq \nu \leq \mu$ and certain continuous piecewise linear functions
$\alpha , \beta :[\nu, \mu ] \rightarrow \R_{+}$ with rational slopes such that
$\beta$ is concave and $\alpha$ is increasing and convex.
\end{theoremB}

When the divisor $D$ is in fact a $\Q$-divisor, the break-points of the
functions $\alpha$ and $\beta$ occur at rational points and the number
$\nu$ must be rational. We also show that the number $\mu$ might be irrational, but it satisfies a quadratic equation over $\Q$. We have not been able to establish which quadratic irrationals 
arise this way: Remark \ref{seshadri} links this problem to the irrationality of 
Seshadri constants.

Theorem B is proved using Zariski decomposition as in 
\cite[Theorem 6.4]{LM08}. More precisely, Theorem B is established via a detailed 
analysis of the variation of Zariski decomposition along a line segment. 
Conversely, we show that all convex bodies as in Theorem B are Okounkov bodies 
of divisors on smooth toric surfaces. 

An example of a non-polyhedral Okounkov body in higher dimensions
was given in \cite[Section 6.3]{LM08}, so no simple characterisation of 
Okounkov bodies along the lines of Theorem B can hold in higher dimensions.
However, it is expected that polyhedral Okounkov bodies are related to
finite generation of rings of sections. In \cite{LM08}, Lazarsfeld and Musta\c t\u a asked if every Mori dream space admits a flag with respect to which the global Okounkov cone is polyhedral. In Section 3 we give two examples of 
Mori spaces (one of which is $\P^2\times\P^2$) equipped with flags with respect to which 
most Okounkov bodies are not polyhedral. The second example has the 
advantage that the shape of the Okounkov body in question is stable under generic deformations of the flag.

\subsection*{Acknowledgments.}
Part of this work was done while the first and the second authors were enjoying
the hospitality of the Universit\'e Joseph Fourier in Grenoble. We would like to
take this opportunity to thank Michel Brion and the Department of Mathematics 
for the invitation. We are grateful to Dave Anderson, Sebastien Boucksom, Jos\'e Gonz\'alez, Shin-Yao Jow, Askold Khovanskii, Rob 
Lazarsfeld and Mircea Musta\c t\u a for many helpful discussions.

\section{Countability of Okounkov bodies}

In this section we prove Theorem A using global Okounkov
 cones. Let $X$ be a smooth projective complex variety of dimension $n$ and let
$Y_{\bullet}$ be an admissible flag on $X$. Let $N^1(X)$ be the 
N\'eron-Severi group of $X$, while $N^1(X)_{\R}$ will denote the 
(finite-dimensional) vector space of numerical equivalence classes of $\R$-divisors.

Consider the  additive sub-semigroup of $\N^{n}\times N^1(X)$ 
\[
\Gamma_{Y_{\bullet}}(X) \deq \Big\{ 
(\nu_{Y_{\bullet}}(D), [L]) \ | \ L\mbox{ a line bundle on X with } D\geq 0 \mbox{ and } \mathcal{O}_X(D)\simeq L\Big\}.
\]
The global Okounkov cone of $X$ with respect 
to $Y_{\bullet}$ is then given by 
\[
\Delta_{Y_{\bullet}}(X) \deq  \text{ closed convex cone generated by }
 \Gamma_{Y_{\bullet}}(X) \mbox{ inside } \R^n\times N^1(X)_\R\ .
\]
Theorem B of \cite{LM08} says that for any big class $\xi\in N^1(X)_{\Q}$
we have that 
\[
 \Delta_{Y_{\bullet}}(X)\cap (\R^n\times\{\xi\}) \equ 
\Delta_{Y_{\bullet}}(X;\xi ) \ .
\]
Thus to prove Theorem A, it is enough to show the following claim, which establishes the countability of the set of global Okounkov cones.

\begin{theorem}\label{okounkovbodies}
There exists a countable set of  closed convex cones $\Delta_i 
\subseteq \R^n\times \R^\rho$ with $i\in\N$ with the property that for any smooth, 
irreducible, projective variety $X$ of dimension $n$ and Picard number $\rho$ 
and any admissible flag $Y_{\bullet}$ on $X$, there is an integral linear isomorphism
\[
\psi_X: \R^\rho \ \rightarrow \ N^1(X)_{\R},
\]
depending only on $X$, such that $(\textup{id}_{\R^n}\times\psi^{-1}_X)(\Delta_{Y_{\bullet}}(X))$ is 
equal to $\Delta_i$ for some $i\in\N$.
\end{theorem}

We say that $\psi_X$ is \emph{integral} if $\psi_X(\Z^\rho)\subset N^1(X)$.

\begin{remark}
In \cite{LM08},  Okounkov bodies were defined in a more general setup.
The subvarieties $Y_i$ were not assumed to be smooth, but merely irreducible,
and smooth at the point $Y_n$. The statement of Theorem A can easily be generalised 
to flags of this form. For this, suppose that Theorem A
holds under the hypothesis that each element of $Y_\bullet$ is smooth.

Consider now a smooth variety $X$ with a flag $Y_{\bullet}$ of irreducible 
varieties, smooth at the point $Y_n$. Choose a proper birational map 
$\mu :X'\rightarrow X$, which is an isomorphism in some neighbourhood of $Y_n$, such that 
the proper transform $Y_i'$ of each $Y_i$ is smooth and irreducible. The flag $Y_{\bullet}'$ is then admissible in our sense and hence for any line bundle $\mathcal{L}$ on $X$ 
there is an $i\in \N$ such that $\Delta_{Y'_{\bullet}}(X';\mu^*\mathcal{L})=\Delta_i$. By Zariski's Main Theorem 
$\mu_*(\mathcal{O}_{X'})=\mathcal{O}_X$ and hence
\[
H^0(X,\mathcal{L}^{\otimes m}) \ = \ H^0(X', 
\mu^{*}(\mathcal{L}^{\otimes m}))
\]
for any $m\in \N$.  Since $\mu$ is an isomorphism in a neighborhood of $Y_n$, it follows
that $\Delta_{Y_{\bullet}}(X;\mathcal{L})= \Delta_{Y_{\bullet}'}(X';\mu^*(\mathcal{L}))=\Delta_i$.
\end{remark}

We now give some definitions and technical prerequisites needed in the proof of 
Theorem~\ref{okounkovbodies}. We set
\[
W \ \deq \ \underbrace{\P^{2n+1} \times \ldots \times \P^{2n+1}}_{\rho \text{ times }}\ .
\]
Note that every line bundle on $W$ has the form
\[
\mathcal{O}_W(\underline{m}) \deq  p_{1}^{*}(\mathcal{O}_{\P^{2n+1}}(m_1))\otimes\ldots \otimes p_{\rho}^{*}(\mathcal{O}_{\P^{2n+1}}(m_{\rho})) 
\] 
for some $\underline{m} \deq (m_1,\ldots ,m_{\rho})\in \Z^{\rho}$, where $p_i : 
W \rightarrow \P^{2n+1}$ is the projection onto the $i$-th factor. For a
projective subscheme $X\subseteq W$ we define its multigraded Hilbert function 
by
\[
P_{X}(\underline{m}) \deq  \chi(X,(\mathcal{O}_W(\underline{m}))|_X), 
\textrm{ for all } \underline{m}\in \Z^{\rho}.
\]
For any projective smooth subvariety $X\subseteq W$ we denote by $\psi_X$ the map
\[
\psi_X: \Z^\rho\rightarrow N^1(X), 
\]
where $\psi_X(\underline{m})= [(\mathcal{O}_W(\underline{m}))|_X]$. 
We also denote the induced map $\psi_X: \R^{\rho}\rightarrow N^1(X)_\R$ by $\psi_X$.  
\begin{proposition}\label{flagfamilies}
Suppose given an $(n+1)$-tuple of numerical functions
$\mathbf{P}=(P_0,\ldots, P_n)$, where $P_i: \mathbb{Z}^\rho\rightarrow \Z$
for all $i$. There exists a quasi-projective scheme $H_{\mathbf{P}}$, 
a closed subscheme $\mathcal{X}_{\mathbf{P}}\subset W\times H_{\mathbf{P}}$ and
a flag of closed subschemes $\mathcal{Y}_{\bullet, \mathbf{P}}: \mathcal{X}_{\mathbf{P}}=\mathcal{Y}_{0}
\supset\mathcal{Y}_{1}\supset\ldots\supset\mathcal{Y}_{n}$ such that 
\begin{enumerate}
\item the induced projection map 
$\phi_i: \mathcal{Y}_{i}\rightarrow H_{\mathbf{P}}$ is flat and 
surjective for all $i$,
\item for all $i$ and all $t\in H_{\mathbf{P}}$ we have that $P_{\mathcal{Y}_{i,t}}=
P_i$,
\item for any projective subvariety $X\subseteq W$ of dimension $n$ and any complete flag of subvarieties 
$X=Y_0\supset Y_1\supset Y_2\supset\ldots \supset Y_n$ such that $P_{Y_i}=P_i$ 
there exists a closed point $t\in H_{\mathbf{P}}$ and an isomorphism $\beta: \mathcal{X}_t
\rightarrow X$ with the property that $\beta (\mathcal{Y}_{i,t})= Y_i$ for all $i$.
\end{enumerate}
\end{proposition}

\begin{proof}
For each $i$, \cite[Corollary 1.2]{HS04} says that there exists a multigraded Hilbert scheme $H_{P_i}$. This is equipped with a flat surjective family $\mathcal{Y}'_i\subset W\times H_{P_i}$ that has the property that for any $Y'_i\subset W$ with $P_{Y'_i}= P_i$ 
there is a $t$ such that $\mathcal{Y}'_{i,t}= Y'_i$.

We consider $H_{P_i}$ and $\mathcal{Y}'_i$  with their reduced scheme 
structure. We now define
\[ 
H_{\mathbf{P}}\subset H_{P_0}\times H_{P_1}\times \ldots \times H_{P_n}
\]
to be given by the incidence relation:  $t=(h_0,\ldots, h_n)\in
H_{\mathbf{P}}$ if and only if $\mathcal{Y}'_{i,h_i}\subset
\mathcal{Y}'_{i-1,h_{i-1}}$ for all $i$. Each element $\mathcal{Y}_i$ of the flag 
$\mathcal{Y}_{\bullet ,\mathbf{P}}$ is defined to be $\mathcal{Y}_i= \pi_i^*(\mathcal{Y}'_i)$, where 
$\pi_i: H_{\mathbf{P}}\rightarrow H_{P_i}$ 
is the projection onto the factor $H_{P_i}$. By definition, 
$\mathcal{Y}_i\subset \mathcal{Y}_{i-1}$ for all $i$ and 
$\mathcal{Y}_i\rightarrow H_{\mathbf{P}}$ is surjective and 
flat because $\mathcal{Y}_i'$ is: condition 1) therefore holds. Condition 2)
 is immediate. By the universal property of the  multigraded Hilbert schemes
$H_{P_i}$, Condition 3) is also satisfied. This completes the proof of
Proposition \ref{flagfamilies}
\end{proof}

\begin{proof}[Proof of Theorem \ref{okounkovbodies}] Let 
$X$ be a smooth irreducible variety of dimension $n$ and Picard number 
$\rho$, equipped with a flag $Y_\bullet$. 
We start by showing that $X$ can then be embedded in $W$ in such a way 
that the induced map of real vector spaces 
$\psi_X$ is an integral isomorphism. Choose $\rho$ very ample line 
bundles $\mathcal{L}_{1,X},\ldots ,
\mathcal{L}_{\rho ,X}$ on $X$ forming a $\Q$-basis of $N^1(X)_{\Q}$. 
As $X$ is smooth, \cite[Theorem 5.4.9]{Sh94} says that for every $i$ 
there is an 
embedding $\alpha_i:X\hookrightarrow \P^{2n+1}$ such that $\mathcal{L}_{i,X}\equ
\alpha_i^*(\mathcal{O}_{\P^{2n+1}}(1))$. We can then embed $X$ in $W$ via
\begin{equation}\label{embedding}
X \stackrel{\Delta}{\longrightarrow} 
\underbrace{X\times \ldots \times X}_{\rho \text{ times }} 
\stackrel{\alpha_1\times\ldots\times \alpha_\rho}{\longrightarrow} W
\end{equation}
where $\Delta$ is the diagonal morphism. Note that 
$\psi_X : \R^\rho \rightarrow  N^1(X)_{\R}$ is an integral 
linear isomorphism by construction. 

Let us now consider the admissible flag $Y_\bullet$ on 
$X$; the multigraded Hilbert functions $P_{Y_i}$ are polynomials with 
rational coefficients. There are therefore only countably many $(n+1)$-tuples of 
numerical functions $\mathbf{P}$ which appear as the multigraded Hilbert function of a 
smooth $n$-dimensional subvariety of $W$, equipped with an admissible flag. By Proposition~\ref{flagfamilies}, there exist countably many 
quasi-projective schemes $T_j$ and closed subschemes $\mathcal{X}_j
\subset W\times T_j$, each equipped with a flag $\mathcal{Y}_{\bullet ,j}$. These families of flags 
have the property that for any smooth irreducible variety $X$ of 
dimension $n$ and Picard number $\rho$ and any 
admissible flag $Y_{\bullet}$ on $X$ there is a closed point $t\in T_j$ for some 
$j$ such that the variety-flag pair  $(\mathcal{X}_{j,t}, \{\mathcal{Y}_{\bullet ,j,t}\})$ is 
isomorphic to the variety-flag pair $(X,\{Y_{\bullet}\})$ and the map 
$\psi_X: \R^\rho\rightarrow N^1(X)_\R$ is an integral isomorphism. We may
without loss of generality consider the schemes $T_j$,
$\mathcal{X}_j$ and $\mathcal{Y}_{i,j}$ with their reduced structure.

Through the rest of the proof of Theorem~\ref{okounkovbodies}, $T$ will be a 
reduced and irreducible  quasi-projective scheme and $\mathcal{X}\subset W\times T$ will be a closed 
subscheme such that the induced projection map
$\phi :  \mathcal{X} \rightarrow  T$ is surjective, flat and 
projective. We suppose given a flag of closed subschemes of $\mathcal{X}$
\[
\mathcal{Y}_{\bullet} \ : \ \ \mathcal{X}= \mathcal{Y}_0 \ \supseteq \ \mathcal{Y}_1 \ \supseteq \ \ldots \ \supseteq \ \mathcal{Y}_{n-1} \ \supseteq \ \mathcal{Y}_n
\]
such that the restriction maps $\phi_i\deq\phi|_{\mathcal{Y}_i} \ : \ 
\mathcal{Y}_i\rightarrow \ T$ are flat, projective and surjective. 
We say that $t\in T$ has an 
\emph{admissible fibre} if the fibre $X_t$ is smooth and irreducible and 
the flag $Y_{t,\bullet}$ is admissible. We will say that $t\in T$ is 
\emph{fully admissible} if it has an admissible fibre and the induced map 
$\psi_{X_t}: \R^\rho\rightarrow N^1(X_t)_\R$ is an isomorphism. With this notation in hand we prove the following proposition.
\begin{proposition}\label{OBrestricted}
Given $T$, $\mathcal{X}$ and $\mathcal{Y}_i$ as above, there exists a 
countable set of convex cones $(\Delta_i)_{i\in\N} \subset\mathbb{R}^n\times 
\mathbb{R}^\rho$ such that for any fully admissible $t\in T$ the cone
$({\rm id}_{\R^n}\times \psi_{X_t}^{-1})(\Delta_{Y_{\bullet,t}}(X_t))$ is equal to $\Delta_i$ for 
some $i$.  
\end{proposition}
\begin{proof}[Proof of Proposition \ref{OBrestricted}]
We consider the $\mathcal{Y}_i$'s with their reduced structure. Since 
the proposition is immediate if there is no fully admissible 
$t$ we may assume at least one such $t$ exists.
 
By induction on the dimension of $T$, it will be enough to prove the existence 
of a non-trivial open subset $U\subseteq T$ such that the conclusions of 
Proposition~\ref{OBrestricted} hold for any fully admissible $t\in U$.  
We can therefore assume $T$ is smooth and there exists a fully admissible 
closed point $t_0\in T$. Since each $\phi_i$ is flat, by \cite[Theorem 12.2.4]{Gr} the set of points 
$t\in T$ such that $Y_{i,t}$ is smooth and irreducible for every $i$ is open 
in $T$. We can therefore assume that all 
$t\in T$ have an admissible fibre.

Since  $X_t$ is smooth for all $t\in T$, the map $\phi$ is smooth by 
\cite[Theorem III.10.2]{Ha77}. Therefore we can  apply 
\cite[Proposition 2.5]{KLM09} or Ehresmann's theorem to deduce that the map
$\psi_{X_t}$ is injective for all $t\in T$.

With this in mind, it will be enough to show that under the above hypotheses, the set 
\[ 
\{ ({\rm id}_{\R^n}\times \psi_{X_t}^{-1}) (\Delta_{Y_{\bullet, t}}(X_t)) \ | \ t\in T \ \}
\] 
is countable. We note further that for any two points $t_1, t_2\in T$ we have that
\[
({\rm id}_{\R^n}\times \psi_{X_{t_1}}^{-1}) (\Delta_{Y_{\bullet, t_1}}(X_{t_1})) \equ 
({\rm id}_{\R^n}\times \psi_{X_{t_2}}^{-1}) (\Delta_{Y_{\bullet, t_2}}(X_{t_2}))
\] 
if  and only if $\Delta_{Y_{\bullet, t_1}}(X_{t_1}; (\mathcal{O}_W(\underline{m}))|_{X_{t_1}})$ is equal to
$\Delta_{Y_{\bullet, t_2}}(X_{t_2}; (\mathcal{O}_W(\underline{m}))|_{X_{t_2}})$ for every 
$\underline{m}\in \Z^\rho$. This will be the case whenever it happens that 
 \begin{equation}\label{E1}
\textup{Im}(\nu_{Y_{t_1,\bullet}}: \ H^0(X_{t_1},
(\mathcal{O}_{W}(\underline{m}))|_{X_{t_1}}) \ 
\rightarrow \ \mathbb{Z}^n)  \equ
\textup{Im}(\nu_{Y_{t_2,\bullet}}\ H^0(X_{t_2},(\mathcal{O}_{W}(\underline{m}))|_{X_{t_2}}) \ \rightarrow \ \mathbb{Z}^n)
\end{equation}
for any $\underline{m}\in \mathbb{Z}^\rho$. Thus it suffices to show that there exists a 
subset $F =\cup F_{\underline{m}} \subseteq T$ consisting of a countable union of proper 
Zariski-closed subsets $F_m \varsubsetneqq T$ such that \eqref{E1} holds 
for every $\underline{m}\in \Z^\rho$ whenever $t_1, t_2\in T\setminus F$. By induction on dim$(T)$, 
this implies Proposition~\ref{OBrestricted}. 

We recall that each morphism $\phi_i:\mathcal{Y}_i
\rightarrow T$ has smooth irreducible fibers, so by  \cite[Theorem III.10.2]{Ha77} $\phi_i$ is smooth. Since $T$ is smooth each $\mathcal{Y}_i$ is smooth and 
hence 
$\mathcal{Y}_{i+1}\subseteq \mathcal{Y}_i$ is Cartier. Thus our family of flags satisfies the conditions of \cite[Theorem 5.1]{LM08}, and consequently,   for any $\underline{m}$ there exists a proper closed subset  $F_{\underline{m}}\subseteq T$ such that the sets 
\begin{equation}\label{open}
 \ \textup{Im}(\nu_{Y_{t,\bullet}}: \ H^0(X_t,(\mathcal{O}_W(\underline{m}))|_{X_t})) 
\ \rightarrow \ \mathbb{Z}^n)
\end{equation}
coincide for all $t\not\in F_{\underline{m}}$. Thus upon setting
$F=\cup F_{\underline{m}}$, $F$ has the properties we seek and 
\[
({\rm id}_{\R^n}\times \psi_{X_{t}}^{-1})(\Delta_{\mathcal{Y}_{t,\bullet}}(X_t)) \ \subseteq \ \R^n\times \R^\rho 
\]
is independent of $t\in T\setminus F$. When $t\in T$ is fully admissible $\psi_{X_t}$ is an isomorphism and this 
completes the proof of Proposition \ref{OBrestricted}.
\end{proof}
By above, any smooth variety of dimension $n$ and Picard number $\rho$ and any admissible flag on it is a fiber in one of the countably many variety-flag pairs $(\mathcal{X}_{j,t}, \{\mathcal{Y}_{\bullet ,j,t}\})$. Thus, by Proposition~\ref{OBrestricted}, we deduce the countability of global Okounkov cones.
\end{proof}

\section{Conditions on Okounkov bodies on surfaces.} 

We turn our attention to  Theorem B which characterises the convex
 bodies arising as Okounkov bodies of big $\R$-divisors on smooth surfaces. 
Whilst we do not characterise them completely, we also establish fairly strong 
conditions on the set of convex bodies which are Okounkov bodies of 
$\Q$-divisors. Our main technical tool will be  Zariski decomposition of divisors. 

Throughout the rest of this section, $S$ will be a smooth surface equipped 
with an admissible flag $(C,x)$, consisting of a smooth curve $C\subseteq S$ and a point $x\in C$, and $D$ will be a pseudo-effective real (or 
rational) divisor on $S$.

Any pseudo-effective divisor $D$ has a  \textit{Zariski decomposition}, 
(the effective case was treated in \cite{Z62} and the pseudo-effective one in \cite{Fu82}; see also \cite[Theorem 14.14]{LB01} for an account of the
proof of this fact). By a Zariski decomposition of  $D$ we mean that $D$ can be 
uniquely written as a sum
\[
D \ = \  P(D) \ + \ N(D)
\]
of $\R$-divisors (or $\Q$-divisors whenever $D$ is such) with the property that $P(D)$ is nef, $N(D)$ is
either zero or  effective with negative definite intersection matrix, and $(P(D).E)=0$ for 
every irreducible component $E$ of $N(D)$. $P(D)$ is called the \textit{positive part} of $D$ and $N(D)$ the \textit{negative part}. 
Another important property of the Zariski decomposition is the minimality of the 
negative part (first proved in \cite{Z62}, c.f. 
\cite[Lemma 14.10]{LB01}). This states that if $D=M+N$, where $M$ is nef and 
$N$ effective, then $N-N(D)$ is effective. 

We prove Theorem B using Lazarsfeld and Musta\c t\u a's description of the 
Okounkov body of a divisor on a surface (\cite[Theorem 6.4]{LM08}) via 
Zariski decomposition. Let $\nu$ be the coefficient of $C$ in the negative 
part $N(D)$ and set
\[ 
\mu \ = \ \mu (D;C) \ = \ \textup{sup}\{ \ t>0 \ | \ D-tC \textup{ is big }\} . 
\]
When there is no risk of confusion we will denote $\mu(D;C)$ by $\mu(D)$. 
For any $t\in [\nu ,\mu ]$ we set $D_t=D-tC$ and write $D_t=P_t+N_t$ for its 
Zariski decomposition. There then exist two continuous functions $\alpha ,
\beta :[\nu ,\mu ]\rightarrow \R_+$ defined as follows
\[ 
\alpha(t)= {\rm ord}_x(N_t|_C), \mbox{ }\mbox{ }\beta(t)={\rm ord}_x(N_t|_C)+
P_t\cdot C
\] 
such that the Okounkov body $\Delta_{(C,x)} (S;D)\subseteq \R^2$ is the region 
bounded 
by the graph of $\alpha$ and $\beta$: 
\[ 
\Delta_{(C,x)} (S;D)=\{ (t,y)\in \mathbb{R}^2 \ | \ \nu\leq t\leq \mu, \alpha(t)\leq 
y\leq \beta(t)\}\ . 
\]
We now set $D'\equ  D-\mu C$: the divisor $D'$ is pseudo-effective by 
definition of $\mu$. For any $t\in[\nu, \mu ]$ we write $s=\mu -t$ and set
\[ 
D'_s\deq D'+sC =D'+(\mu-t)C =D-tC.
\]
It turns out to be more useful to consider the line segment $\{D_t \ | \ t\in [\nu, \mu]\}$ in the form $\{D'_s \ | \ s\in [0, \mu-\nu]\}$. Let $D'_s=
P'_s+N'_s$ be the Zariski decomposition of $D'_s$: the following proposition 
examines the variation $N'_s$ as a function of  $s\in [0, \mu -\nu ]$.

\begin{proposition}\label{prop1} 
The function $s\mapsto N'_s$ is decreasing on the interval 
$[0, \mu-\nu]$, i.e. for each $0\leq s'< s\leq \mu-\nu$ the divisor 
$N'_{s'}-N'_s$ is effective. If $n$ is the number of irreducible components of $N'_0$, then there is a partition $(p_i)_{0\leq i\leq k}$ of the interval $[0, \mu-\nu]$, for some $k\leq n$, and there exist divisors $A_i$ and $B_i$ with $B_i$ rational 
such that $N'_s=A_i+sB_i$ for all $s\in [p_i,p_{i+1}]$.
\end{proposition} 
\begin{proof}
Let $C_1,\ldots, C_n$ be the irreducible components of $\Supp (N'_0)$. 
Choose real numbers $s',s$ such that $0\leq s'< s\leq \mu-\nu$. We 
can then write
\[ 
P'_{s'} \equ D'_{s'}-N'_{s'} \equ  (D'_s-(s-s')C)-N'_{s'} \equ D'_s - 
((s-s')C+N'_{s'}).
\] 
As $P'_{s'}$ is nef and the negative part of the Zariski decomposition is 
minimal, the divisor $(s-s')C+N'_{s'}-N'_s$ is effective and it remains only to show that $C$
is not in the support of $N'_s$ for any $s\in [0, \mu-\nu]$. If $C$ were in 
the support of $N'_s$ for some $s\in [0, \mu-\nu]$, then for any 
$\lambda >0$ the Zariski 
decomposition of $D'_{s+\lambda}$ would be $D'_{s+\lambda}=P'_s+(N'_s+\lambda C)$.
In particular, $C$ would be in the support of  $N'_{\mu-\nu}$, contradicting  
the definition of $\nu$. 

Rearranging the $C_i$'s, suppose that the support of $N'_{\mu-\nu}$ consists of 
$C_{k+1}, \ldots, C_n$. Let 
\[
p_i \deq  \sup \{ s\ | \ C_i\subseteq \Supp(N'_s)\} \textup{ for all } i=1\ldots k\ .
\]
Without loss of generality, suppose $0=p_0< p_1\leq \ldots \leq p_{k-1}\leq p_k\leq \mu 
-\nu$. We will show that $N'_s$ is linear on $[p_i, p_{i+1}]$ for this choice of
$p_i$'s. By the continuity of the Zariski decomposition (see \cite[Proposition 1.14]{BKS04}), it 
is enough to show that $N'_s$ is linear on the open interval $(p_i,p_{i+1})$. If
$s\in (p_i, p_{i+1})$ then the support of $N'_s$ is contained in 
$\{C_{i+1},\ldots ,C_{n}\}$, and $N'_s$ is determined uniquely by the equations
\[
N'_s\cdot C_j \equ (D'+sC)\cdot C_j, \textup{ for } i+1\leq j\leq n\ .
\] 
As the intersection matrix of the curves $C_{i+1},\ldots ,C_{n}$ is 
non-degenerate, there are unique divisors  $A_i$ and $B_i$ supported on 
$\cup_{j=i+1}^n C_j$ such that
\[
A_i\cdot C_j \equ D'\cdot C_j \textup{ and } B_i\cdot C_j= C\cdot C_j \textup{ for all }i+1\leq j\leq n\ .
\] 
Note that $B_i$ is a rational divisor and it follows that 
$N'_s= A_i+s B_i$ for any $s\in (p_i, p_{i+1})$.
\end{proof} 

\begin{proof}[Proof of Theorem B] 
Theorem 6.4 of \cite{LM08} implies that $\alpha$ is convex, $\beta$ is concave 
and $\alpha\leq\beta$. It follows from Proposition~\ref{prop1} that $\alpha$ and
$\beta$ are piecewise linear with only finitely many break-points. And finally 
$\alpha$ is an increasing function of $t$ by Proposition \ref{prop1}, because 
$N_t= N'_{\mu-s}$ and $\alpha (t)= {\rm ord}_x(N_t|C)$. This proves that any 
Okounkov body has the required form.

Conversely, we show that a polygon as in Theorem B is the Okounkov body of a 
real $T$-invariant divisor on some toric surface\footnote{We thank Sebastien 
Boucksom, who suggested using toric surfaces, to 
replace a more complicated example using iterated blow-ups of $\P^2$.}. This
section of the proof is based on Proposition 6.1 in \cite{LM08}
which characterises the Okounkov body of a $T$-invariant divisor with respect 
to a $T$-invariant flag in a toric variety in terms of the polygon
associated to $T$ in the character lattice $M_\Z$ associated to $S$.

Let $\Delta\subseteq \R^2$ be a polygon of the form given in Theorem B. As 
$\alpha$ is increasing we can assume after translation that 
$(0,0)\in \Delta \subseteq \R^2_+$. We identify $\R^2$  with the  
vector space $M_\R$ associated to a character lattice $M_{\Z}= \Z^2$. Let $E_1,
\ldots ,E_m$ be the edges of $\Delta$. Considering that $\alpha$ and $\beta$ have rational slopes, for each edge $E_i$ choose a 
primitive vector $v_i\in N_{\Z}$ normal to $E_i$ in the direction of the 
interior of $\Delta$, where $N_{\R}$ is the dual of $M_{\R}$. We can then write
\[
\Delta \ = \ \{\ u\in M_\R \ | \ \langle u, v_i\rangle+a_i \geq 0\mbox{ for 
all } i=1\ldots m \ \}
\]
for some positive real $a_i$'s. After adding additional vectors 
$v_{m+1},\ldots, v_r$ we can assume that the set $\{ v_1,\ldots, v_m\}$ has the
following properties.
\begin{enumerate}
\item The toric surface $S$ associated to the complete fan $\Sigma$ which is 
defined by the rays $\{\R_+\cdot v_1,\ldots ,\R_+\cdot v_r\}$ is smooth.
\item None of the $v_{i}$s lie in the interior of the first quadrant.
\item for some $i_1,i_2 \in \{ 1,\ldots ,m\}$ we have that $v_{i_1}= 
\left( \begin{array}{c} 1 \\ 0 \end{array} \right)$ and $v_{i_2}= 
\left( \begin{array}{c} 0 \\ 1 \end{array} \right)$
\end{enumerate}
Condition $(2)$ is possible because $\alpha$ is increasing. Since $\Delta$ is 
compact there exist real numbers $a_{m+1},\ldots ,a_r\in \Q_+$ such that
\[
\Delta \ = \ \{\ u\in M_\R \ | \ \langle u, v_i\rangle +a_i\geq 0\mbox{ for all 
} i=1\ldots r \ \}
\]
Condition $(3)$ implies that we can choose $a_{i_1}=a_{i_2}=0$. The general 
theory of toric surfaces now tells us that 
each $v_i$ represents a $T$-invariant divisor $D_i$ 
on $S$ and on setting $D=\Sigma a_iD_i$ the polytope $P(D)\subseteq M_{\R}$
 associated to $D$ is equal to $\Delta$. We choose on $S$ the flag consisting 
of the curve $C=D_{i_1}$ and the point $\{ x\}=D_{i_1}\cap D_{i_2}$. The curve 
$C$ is smooth and the intersection $D_{i_1}\cap D_{i_2}$ is a point because of 
conditions $(1)$ and $(2)$. By 
\cite[Proposition 6.1]{LM08}, the Okounkov body 
$\Delta_{(C,x)}(S;D)$ of $D$ with respect to the flag $(C,x)$ is equal to 
$\psi_{\R}(P(D))$ where the map $\psi_{\R}: M_{\R}\rightarrow \R^2$ is defined as 
follows
\[
\psi_{\R} (u) \ = \ (\langle u,v_{i_1}\rangle,\langle u,v_{i_2}\rangle) \textup{ for any } u\in M_{\R}.
\]
In our case $\psi_{\R}\equiv \textup{id}_{\R}$, so $\Delta_{(C,x)}(S;D)=
P(D)=\Delta$ by construction. This completes the proof of Theorem B.
\end{proof}

It is now natural to ask the following question: which of these polygons is
the Okounkov body of a rational divisor? The above toric-surface construction implies 
that any polygon of the form considered in Theorem B which is 
given by rational data is the Okounkov body of a rational divisor. The next 
result provides  a partial converse to the effect that the rationality
of the divisor imposes strong rationality conditions on the  points of the
Okounkov body. 

\begin{proposition}\label{prop2}
Let $S$ be a smooth projective surface, $D$  a big rational divisor on $S$ and 
$(C,x)$ be an admissible flag on $S$. Then 
\begin{enumerate}
\item all the vertices of the polygon $\Delta(D)$ contained in the 
set $\{[\nu,\mu)\times \R\}$ have rational coordinates.
\item $\mu (D)$ is either rational or satisfies a quadratic equation over $\Q$. 
\item If an irrational number $a>0$ satisfies a quadratic equation over $\Q$ 
and the conjugate $\overline{a}$ of $a$ over $\Q$ is strictly larger than $a$,
 then there exists a smooth, projective surface $S$, an ample $\Q$-divisor $D$ 
and an admissible flag on $S$ such that $\mu(D)=a$.
\end{enumerate}
\end{proposition}

\begin{proof}
The number $\nu$ is rational because the positive and negative parts of the Zariski decomposition 
of a $\Q$-divisor are rational: it follows that $\alpha (\nu )$ and
$\beta (\nu )$ are rational. It follows from the proof of 
Proposition~\ref{prop1} that  the break-points of $\alpha$ and $\beta$ occur at 
points $t_i$ which are intersection points between the line $D-tC$ and 
faces of the Zariski chamber decomposition of the cone of big divisors  
\cite{BKS04}. However, it is proved in  \cite[Theorem 1.1]{BKS04} that this decomposition is
locally finite rational polyhedral, and hence  the break-points of 
$\alpha$ and $\beta$ occur at rational points.

For $(2)$, notice that the volume $\textrm{vol}_X(D)$, which is half of the area of the 
Okounkov polygon $\Delta(D)$, is rational (see \cite[Corollary 2.3.22]{La04}). As the  slopes and intermediate breakpoints of $\Delta(D)$ are rational, the equation  computing the area of $\Delta (D)$ gives a quadratic equation for $\mu(D)$ with
 rational coefficients. Note that if $\mu$ is irrational  then one edge of 
the polygon $\Delta (D)$ must sit on the vertical line $t=\mu$.

The final part of the proposition follows from a
 result of Morrison's \cite{Morrison} which states 
that any even integral quadratic form $q$ of signature $(1,2)$ occurs as the 
self-intersection form of a K3 surface $S$ with Picard number $3$. An argument 
of Cutkosky's \cite[Section 3]{C} shows that if the coefficients of the form are 
all divisible by $4$, then the pseudo-effective and nef cones of $S$ coincide 
and are given by 
\[
\{ \alpha\in N^1(S)\,|\, (\alpha^2)\geq 0\, ,\, (h\cdot\alpha)>0 \}
\]
for any ample divisor $h$ on $S$. If $D$ is an ample divisor and 
$C\subseteq S$ an irreducible curve (not in the same class as $D$), then the function $f(t)\deq ((D-tC)^2)$ 
has two positive roots and $\mu(D)$ with respect to $C$ is equal to the 
smaller one, i.e.
\[
 \mu(D) \equ \frac{(D\cdot C)-\sqrt{(D\cdot C)^2-(D^2)(C^2)}}{(C^2)}\ .
\]
Since we are only interested in the roots of $f$ we can start with any 
integral quadratic form of signature $(1,2)$ and multiply it by $4$. Hence we 
can exhibit any number with the required properties as $\mu(D)$ for suitable 
choices of the quadratic form, $D$, and $C$.
\end{proof}
\begin{remark}\label{seshadri}
It was mentioned in passing in \cite{LM08} that the knowledge of all Okounkov bodies 
determines Seshadri constants. In the surface case, there is a link between the
irrationality of $\mu$ for certain special forms of the flag and that of 
Seshadri constants. Let $D$ be an ample divisor on $S$, and let 
$\pi:\widetilde{S}\to S$ be the blow-up of a point $x\in S$ with exceptional 
divisor $E$. Then the Seshadri constant of $D$ at $x$ is defined by
\[
\epsilon (D,x)\deq \sup\{t\in \R\,|\, \pi^*(D)-tE \textrm{ is nef in } 
\widetilde{S}\}.
\]
We note that if $\epsilon (D,x)$ is irrational, then $\epsilon(D,x)=
\mu(\pi^*(D))$ with respect to any flag of the form $(E,y)$. Indeed, the 
Nakai--Moishezon criterion implies that either there is a curve 
$C\subseteq \widetilde{S}$ such that $C\cdot(\pi^*(D)-\epsilon E)=0 $ or 
$((\pi^*(D)-\epsilon E)^2)=0$. But since $C\cdot \pi^*(D)$ and 
$C\cdot E$  are both rational, $C\cdot(\pi^*(D)-\epsilon E)=0 $ is impossible if
$\epsilon$ is irrational. Therefore  $(\pi^*(D)-\epsilon E)^2=0$, hence 
$\pi^*(D) -\epsilon E$ is not big and therefore $\epsilon=\mu\notin\Q$. 
\end{remark}

\section{Non-polyhedral Okounkov bodies}

In this section we will give two examples of non-polyhedral Okounkov bodies of
divisors on Mori dream space varieties, thereby showing in particular that 
ample divisors can nevertheless have non-polyhedral Okounkov bodies. The first example is Fano; the second one is not, 
but has the advantage that the non-polyhedral shape of Okounkov bodies is
stable under generic perturbations of the flag.

\begin{proposition}\label{slices}
Let $X$ be a smooth projective variety of dimension $n$ equipped with an admissible flag $Y_{\bullet}$. 
Suppose that $D$ is a divisor such that $D-sY_{1}$ is ample. Then we have the following
lifting property
\[ \Delta_{Y_{\bullet}}(X;D)\cap \big(\{s\}\times \mathbb{R}^{n-1}\big)= 
\Delta_{Y_{\bullet}}(Y_1;(D-sY_{1})|_{Y_{1}}).\]  
In particular, if
$\overline{\textup{Eff}}(X)_{\mathbb{R}} = \textup{Nef}(X)_{\mathbb{R}}$
then on setting $\mu (D;Y_1)=\textup{sup}\{ \ t>0 \ | \ D-tY_{1}\mbox{ ample }\}$ we have that the Okounkov body
$\Delta_{Y_{\bullet}}(X;D)$ is the closure in $\R^n$ of the following set
\[ 
\{ (s,\underline{v}) \ | \ 0\leq s< \mu (D;Y_1), \underline{v}\in
\Delta_{Y_{\bullet}}(Y_1;(D-sY_{1})|_{Y_{1}}\}
\]
\end{proposition}

\begin{proof}
In order to prove the lifting property we will use \cite[Theorem 4.26]{LM08}, 
which in our context states that
\[
\Delta_{Y_{\bullet}}(X;D)\cap \big(\{s\}\times \mathbb{R}^{n-1}\big) \ = \ \Delta_{Y_{\bullet}}( X|Y_1,D-sY_1)
\]
where the second body is the restricted Okounkov body defined in \cite[Section 2.4]{LM08}. Hence it is enough to show that
\begin{equation}\label{restricted}
\Delta_{Y_{\bullet}}(X|Y_1, D-sY_1) \ = \ \Delta_{Y_{\bullet}}(Y_1,(D-sY_1)|_{Y_1})\ ,
\end{equation}
We will prove this for $s\in \Q_+$, as the general case follows from the 
continuity of slices of Okounkov bodies. Combining 
\cite[Theorem 4.26]{LM08} and \cite[Proposition 4.1]{LM08} we obtain that the 
restricted Okounkov body satisfies the required homogeneity condition, i.e.
\[
\Delta_{Y_{\bullet}}(X|Y_1, p(D-sY_1)) \ = \ p \Delta_{Y_{\bullet}}(X|Y_1,(D-sY_1) \text{ for all } p\in \mathbb{N} \ .
\]
The construction of restricted Okounkov bodies tells us that  (\ref{restricted}) will 
follow if one can check that $H^1(X,m(p(D-sY_1)-Y_1))=0$ for sufficiently large divisible 
$p,m\in\N$. As $D-sY_1$ is an ample divisor, this follows from Serre vanishing.
\end{proof}

\begin{corollary}\label{non-polyhedra}
Let $X$ be a smooth three-fold and $Y_{\bullet}=(X,S,C, x)$ an admissible flag 
on $X$. Suppose that $\overline{\textup{Eff}}(X)_{\mathbb{R}} = 
\textup{Nef}(X)_{\mathbb{R}}$ and $\overline{\textup{Eff}}(S)_{\mathbb{R}} = 
\textup{Nef}(S)_{\mathbb{R}}$. The Okounkov body of any ample divisor $D$ with 
respect to the admissible flag $Y_{\bullet}$ can be described as follows
\[
\Delta_{Y_{\bullet}}(X;D) \ = \ \{(r,t,y)\in\R^3 \ | \ 0\leq r \leq \mu (D;S), 0\leq t \leq f(r), 0\leq y\leq g(r,t) \ \} 
\]
where $f(r) = \textup{sup} \{ \ s>0 \ | \ (D-rS)|_S-sC \textup{ is ample }\}$ and $g(r,t)=(C.(D-rS)|_S)-t(C^2)$. 
(All  intersection numbers in the above formulae are defined with respect to
the intersection form on $S$.)
\end{corollary}

\begin{remark} \label{badshapes}
$(1)$ Corollary \ref{non-polyhedra} follows by combining Proposition~\ref{slices} and the 
description of the Okounkov body of divisors on surfaces given in \cite[Theorem 6.4]{LM08}.\\
$(2)$ In the context of Corollary~\ref{non-polyhedra} the data of the function
$f:[0,\mu (D;S)]\rightarrow \R_+$ can force the associated Okounkov bodies to 
be non-polyhedral. 
Note that $f(r)$ is the real number such that $(D-rS)|_S-f(r)C$ lies on the 
boundary of the pseudo-effective cone of $S$, which  under our 
assumptions coincides 
with the nef cone. The graph of $f(r)$ is therefore 
(an affine transformation of) the curve obtained by intersecting the boundary 
of $\textup{Nef}(S)_{\R}$ with the plane passing through $[D|_S]$, 
$[(D-S)|_S]$ and $[D|_S-C]$ inside the vector space $N^1(S)_{\R}$. If
the Picard group of $S$ has dimension at least three and 
the boundary of the nef cone of $S$ be defined by quadratic rather than
linear equations then this intersection will typically be a conic curve, not
piecewise linear.
\end{remark}

\begin{example}[Non-polyhedral Okounkov body on a Fano variety]
We set $X=\P^2\times\P^2$ and let $D$ be a divisor in the  linear series
$\mathcal{O}_{\P^2\times\P^2}(3,1)$. We set
\[
Y_{\bullet} \ : \ Y_0=\P^2\times\P^2 \ \supseteq \ Y_1=\P^2\times E \ \supseteq \ Y_2=E\times E\ \supseteq \ Y_3=C \ \supseteq \ Y_4=\{\text{pt}\}
\]
where $E$ is a general elliptic curve. Since  $E$ is general we have that
\[
\overline{\textup{Eff}}(E\times E)_{\mathbb{R}}=\textup{Nef}(E\times E)_{\mathbb{R}}=
\{ (x,y,z)\in \mathbb{R}^3 \ | \ x+y+z \geq 0, xy+xz+yz\geq 0\}
\]
under the identification 
\[
\R^3\rightarrow N^1(E\times E)_{\R}, \ (x,y,z)\rightarrow xf_1+yf_2+ z\Delta_E,
\]
where $f_1=\{\mbox{pt}\}\times E$, $f_2=E\times \{\mbox{pt}\}$ and $\Delta_E$ is the
diagonal divisor. Let $C\subseteq E\times E$ be a smooth general curve in 
the complete linear series $|f_1+f_2+\Delta_E|$ and let $Y_4$ be a general 
point on $C$. To prove that the Okounkov body 
$\Delta_{Y_{\bullet}}(X; D)$ is not polyhedral it will be enough to prove
that the slice $\Delta_{Y_{\bullet}} (X;D)\cap \{ 0\times \mathbb{R}^{2}\}$ is 
not polyhedral. Since 
$\overline{\textup{Eff}}(\P^2\times\P^2)_{\mathbb{R}}=
\textup{Nef}(\P^2\times\P^2)_{\mathbb{R}}$, Proposition \ref{slices} applies and 
it will be enough to show that
$\Delta_{Y_ {\bullet}}(Y_1;\mathcal{O}_{Y_1}(D))$ is not polyhedral.

The threefold $Y_1=\mathbb{P}^2\times E$ is homogeneous, so 
its nef cone is equal to its pseudo-effective cone: this cone is bounded by the rays 
$\mathbb{R}_+[\text{line}\times E] \ \text{ and } \ \mathbb{R}_+[\P^2\times \{\text{pt}\}]$.
We note that hypotheses of Corollary \ref{non-polyhedra} therefore apply to 
$Y_1$ equipped with the flag $(Y_2, Y_3, Y_4)$.  

Using the explicit description given above of $\textup{Nef}(Y_1)_{\R}$, we see that 
$\mu (\mathcal{O}_{Y_1}(D),Y_2)=1$. A simple calculation gives us
\[
g(r,t) \ = \ (C.(D|_{Y_1}-rY_2)|_{Y_2})-t(C^2) \ = \ 24-18r-6t.
\]
Let us now consider 
\begin{eqnarray}
f(r) & = & \text{sup}\{ s>0 \ | \ (D-rY_2)|_{Y_2}-sC \text{ is ample}\} 
\nonumber\\
& = & \text{sup}\{ s>0 \ | \ (9-9r-s)f_1+(3-s)f_2-s\Delta_E \text{ is ample }\}.\nonumber
\end{eqnarray}
After calculation, we see that for positive $s$ the divisor 
$(9-9r-s)f_1+(3-s)f_2-s\Delta_E$ is ample if and 
only if $s < (4-3r-\sqrt{9r^2-15r+7})$. Corollary \ref{non-polyhedra} therefore tells us that
the Okounkov body of $D$ on $Y_1$, $\Delta_{Y_{\bullet}}(Y_1;D)$, has the following description
\[ 
\{ (r,t,y)\in\mathbb{R}^3 | 0\leq r\leq 1, 
0\leq t\leq 4-3r-\sqrt{9r^2-15r+7}, 0\leq y\leq 24-18r-6t \}.
\]
As this body is non-polyhedral, the same can be said about the Okounkov body $\Delta_{Y_{\bullet}}(X;D)$.
\end{example}

In the following, we give an example of a Mori dream space such that the 
Okounkov body of a general ample divisor is non-polyhedral and remains so 
after generic deformations of the flag in its linear equivalence class. Our 
construction is 
based heavily on an example of Cutkosky's~\cite{C}. Cutkosky considers a K3 
surface $S$ whose N\'eron-Severi space $N^1(S)_\R$ is isomorphic to  
$\mathbb{R}^3$ with the lattice $\mathbb{Z}^3$ and the intersection form 
$q (x,y,z)= 4x^2-4y^2-4z^2$. Cutkosky shows that 
\begin{enumerate} 
\item The divisor class on $S$ represented by the vector $(1,0,0)$ corresponds 
to the class of a very ample line bundle, which embeds $S$ in $\P^3$ as a 
quartic surface. 
\item The nef and pseudo-effective cones of $S$ coincide, and a vector 
$(x,y,z)\in\R^3$ represents a nef (pseudo-effective) class if it satisfies the 
inequalities
\[ 
4x^2-4y^2-4z^2\geq 0\ ,\  x \geq 0\ . 
\] 
\end{enumerate}
We consider the surface $S\subset \mathbb{P}^3$, and the pseudo-effective 
classes on $S$ given by $\alpha=(1,1,0)$ and $\beta=(1,0,1)$. By Riemann-Roch 
we have that $H^0(S,\alpha)\geq 2$ and $H^0(S, \beta)\geq 2$, so both $\alpha$ 
and $\beta$, being extremal rays in the effective cone, are classes of 
irreducible moving curves. Since 
$\alpha^2=\beta^2=0$, both these families are base-point free,  and it follows 
from the  base-point free Bertini theorem (see page 109 in \cite{flenneretal}) 
that there are smooth irreducible curves $C_1$ and $C_2$ representing $\alpha$ 
and $\beta$ respectively, which are elliptic by the adjunction formula. We may 
assume that $C_1$ and $C_2$ meet transversally in $C_1\cdot C_2=4$ points. 

Our threefold $Z$ is constructed as follows. Let $\pi_1: Z_1\rightarrow 
\mathbb{P}^3$ be the blow-up along the curve $C_1\subseteq \P^3$. We then
define $Z$ to be the blow up of the strict transform 
$\overline{C}_2\subseteq Z_1$ of the curve
$C_2$. Let $\pi_2: Z\rightarrow Z_1$ be the second blow-up and $\pi$ the composition $\pi_1\circ \pi_2: Z\rightarrow \mathbb{P}^3$. We 
denote by $E_2$ the exceptional divisor of $\pi_2$ and by $E_1$ the strict 
transform of the exceptional divisor of $\pi_1$ under $\pi_2$. We have the 
following proposition.
\begin{proposition}\label{MDS} 
The variety $Z$ defined above is a Mori dream space, and $-K_Z$ is effective. 
Given any two ample divisors on $Z$, $L$ and $D$, 
such that the classes $[D], [L], [-K_Z]$ are linearly
independent in $N^1(Z)_{\R}$, the Okounkov body $\Delta_{Y_{\bullet}}(X;D)$
 is non-polyhedral with respect to any admissible flag 
$(Y_1, Y_2, Y_3)$ such that $\mathcal{O}_Z(Y_1)=-K_Z$, ${\rm Pic}(Y_1)=
\langle H, C_1, C_2\rangle$  and 
$\mathcal{O}_{Y_1}(Y_2)=L|_{Y_1}$, where $H$ is the pullback of a hyperplane section of $\P^3$ by the map $\pi$. 
\end{proposition}
\begin{remark}
The advantage of this example over the previous one is that it does not depend 
on a choice of flag elements which are exceptional from a Noether-Lefschetz 
point of view. (Note that by standard Noether-Lefschetz arguments the 
condition that $\textup{Pic}(Y_1)=\langle H, C_1, C_2\rangle$ holds for any
very general $Y_1$ in $|-K_Z|$.)

In particular, the previous example depended upon the fact that $Y_2$ had a 
non-polyhedral nef cone, which in this case was possible only because
$Y_2$ had Picard group larger than that of $X$: moreover, it was necessary to
take $Y_3$ to be a curve not contained in the image of the Picard group of 
$Y_1$.
It is to a certain extent less surprising that choosing flag elements in
$\textup{Pic}(Y_i)$ that do not arise by restriction of elements in 
$\textup{Pic}(Y_{i-1})$ should lead to bad behaviour in the Okounkov body. There
does not seem to be any reason why the fact that $X$ is Fano should influence
the geometry of the boundary of the part of the nef cone of $Y_i$ which does not
arise by restriction from $X$.

Moreover, such behaviour cannot be general, so there is little hope of using such
examples to construct a counter example to \cite[Problem 7.1]{LM08}.
\end{remark}
\begin{proof} 
We start by proving that $Z$ is a Mori dream space ($-K_Z$ is immediately
effective, since $\overline{S}$ is a section of $-K_Z$). 
By \cite[Corollary 1.3.1]{BCHM}, it is enough to find an effective big 
divisorial log terminal divisor $\Delta$ on $Z$ such that $-K_{Z}-\Delta$ is 
ample. The existence of such a $\Delta$ will follow if we can show that 
$-K_{Z}$ is big and nef. Indeed, there then exists an effective divisor $E$ 
such that 
$-K_{Z}-\epsilon E$ is ample for any sufficiently small $\epsilon$. We are
then done on setting $\Delta=\delta(-K_{Z})+\epsilon E$ for any sufficiently 
small $\delta$ and $\epsilon$. 

Let's show that $-K_{Z}$ is nef. The first idea we need is to prove that any base point of 
$\mathcal{O}_Z(-K_Z)$ must be contained in $\pi^{-1}(C_1\cap C_2)$. Note 
that $\pi_*\OO_Z(-K_{Z}) \equ \mathcal{O}_{\P^3}(4)\otimes 
\mathcal{I}_{C_1+C_2}$. We will now show that each $C_i$ is a complete intersection of two quadrics.
Note first that $C_i\subseteq \P^3$ is non-degenerate, since $\mathcal{O}_S(H-C_i)$ has 
self-intersection on $S$ equal to $-4$, and is therefore non-effective. 
We now consider the following exact sequence
\[ 0\rightarrow H^0(\mathcal{I}_{C_i}(2))\rightarrow H^0(
\mathcal{O}_{\mathbb{P}^3}(2))\rightarrow 
H^0(\mathcal{O}_{C_i}(2))\rightarrow .\]
Note that ${\rm dim} (H^0(\mathcal{O}_{\P^3}(2)))=10$ and ${\rm dim}(H^0(\mathcal{O}_{C_i}(2)))=8$
by Riemann-Roch. It follows that ${\rm dim}(H^0(\mathcal{I}_{C_i}(2)))\geq 2$,
so we can find linearly independent quadrics, $P_{i},Q_{i}$, which vanish along
$C_i$. As $C_i\subseteq \P^3$ is nondegenerate and of degree $4$, it must be 
the complete intersection of $P_{i}$ and $Q_{i}$. The pull-back to $Z$ of any
one of the polynomials $P_{1}P_{2},P_{1}Q_{2},Q_{1}P_{2},Q_{1}Q_{2}$ 
gives a section of $\mathcal{O}_Z(-K_Z)$,  so all the base points of
$\mathcal{O}_Z(-K_Z)$ are included in $\pi^{-1}(C_1\cap C_2)$.

To prove that $-K_Z$ is nef, it is therefore enough to check that the 
intersection of $-K_Z$ with any curve contained in $\pi^{-1}(C_1\cap C_2)$ is 
positive. Set $C_1\cap C_2=\{p_1,p_2,p_3,p_4\}$, and let $R_1$ (resp. $R_2$) be
the class of a curve  in the ruling of $E_1$ (resp. $E_2$). For any $i$ the set
$\pi^{-1}(p_i)$ is then the union of two  irreducible curves, one of class 
$R_2$ and the other of class $R_1-R_2$.  We have that 
$R_1\cdot H= R_2\cdot H= R_1\cdot E_2= R_2\cdot E_1=0$ and 
$R_1\cdot E_1=-1$, $R_2\cdot E_2=-1$. In particular, $-K_{Z}\cdot R_2=1$ 
and $-K_{Z}\cdot (R_1-R_2)=0$, so $-K_{Z}$ is nef (but not ample). 

It only remains  to prove  that $-K_Z$ is big. More explicitly, we show 
that the image of $\mathbb{P}^3$  under the rational map 
\[ 
\phi: \mathbb{P}^3\rat \mathbb{P}^4\ , \phi=[F:P_{1}P_{2}:P_{1}Q_{2}:
Q_{1}P_{2}:Q_{1}Q_{2}] 
\] 
is three-dimensional.  Here  $F$ is the polynomial defining the surface 
$S\subseteq \P^3$ and is hence an element of $H^0(\mathcal{O}_{\P^3}(4)\otimes\mathcal{I}_{C_1+C_2})$. 

We start by checking that the image of the restricted map $\phi|_S$ has 
dimension two. Observe  that $\phi|_S$ can be factored as
\[
f\circ (\phi_1\times\phi_2) \ : \ S \dashrightarrow \P^1\times\P^1 \rat\P^4\ ,
\] 
where $f([a:b],[c:d])=[0:ac:ad:bc:bd]$ and $\phi_i=[P_i:Q_i]$. The image of $f$ is of dimension $2$, 
thus it is enough to show that $\phi_1\times\phi_2$ is generically surjective. Both $P_1$ and $Q_1$ 
vanish on $S$ only along $C_1$, thus the general fiber of $\phi_1$ is in the 
class $(2,0,0)-(1,1,0)=(1,-1,0)$ and likewise the general fiber of $\phi_2$ is 
$(1,0,-1)$. Since $(1,0,-1)\not\geq (1,-1,0)$ and $(1,-1,0)\not\geq (1,0,-1)$
in $N^1(S)$, and $\phi_1$ and $\phi_2$ are individually generically surjective,
$\phi_1\times\phi_2$ is also generically surjective. The image of $\phi|_S$ is 
therefore two dimensional. 

It follows that either ${\rm Im}(\phi)$ is three dimensional or 
${\rm Im}(\phi)\subset \overline{{\rm Im}(\phi_S)}$. But if $p\not \in S$ 
then $F(x)\neq 0$ so $\phi(p)\not\in \overline{{\rm Im}(\phi_S)} $. 
Thus the image of $\phi$ is three dimensional and $-K_Z$ is big.

We now show that if $D$ and $L$ satisfy the given independence condition and
$\mathcal{O}_Z(Y_1)= -K_Z$ then $\Delta_{Y_{\bullet}}(Z;D)$ is non-polyhedral.
We start by proving that $V$, the space spanned by $\{ H, C_1, C_2\}$ on $Y_1$
has the same properties as $N^1(S)_{\R}$. For this notice that $Y_1$ is the strict transform of a smooth K3 surface containing 
both $C_1$ and $C_2$ and by assumption ${\rm Pic}(Y_1)=\langle H,C_1, C_2
\rangle$. Also, we have the following equalities of intersection numbers
\[
\langle C_1,C_2\rangle_{Y_1}=\langle C_1,C_2\rangle_S, \langle H,C_1\rangle_{Y_1}=\langle H,C_1\rangle_S, \langle H,C_2\rangle_{Y_1}=\langle H,C_2\rangle_S
\] 
so it remains true that for any integral class $C$ on $Y_1$ we have that
$4| C^2$. In particular, this implies there are no effective irreducible classes
on $Y_1$ with negative self-intersection so
\[ 
\overline{\rm Eff}(Y_1)= {\rm Nef}(Y_1)=\{ v\in {\rm Pic}(Y_1) \ | \ \langle v,v\rangle\geq 0,\langle v,H\rangle\geq 0\}.
\]
For small enough $r$ we have that $D-r Y_1$ is ample. It follows
by Proposition \ref{slices} that for any small enough $r$
\[ 
\Delta_{Y_{\bullet}}(Z;D)\cap \big(\{ r\}\times \mathbb{R}^2\big)=\Delta_{Y_{\bullet}}(Y_1, (D-r Y_1)|_{Y_1}).
\]
We now set
\[ 
f(r)= {\rm max}\{s \ | \ \Delta_{Y_{\bullet}}(Z;D)\cap \big((r,s)\times \mathbb{R}\big)\neq \emptyset\}
\]
and note that $f$ is piece-wise linear if $\Delta_{Y_{\bullet}}(Z;D)$ is a 
polyhedron. We note that for small values of $r$ we have, 
by the explicit description of Okounkov bodies of surfaces, that
\[
f(r) = \textup{sup} \{ \ s>0 \ | \ (D-rY_1)|_{Y_1}-sY_2 \in 
\overline{\textup{Eff}}(Y_1)_{\R}\}\ .
\]
But, as explained in Remark \ref{badshapes}, the graph of $f$ is then
an affine transformation of the intersection of the cone
$\overline{\rm Eff}(Y_1)$ with the plane passing through $D$, $D-Y_1$
 and $D-Y_2$. By hypothesis this plane does not pass thorough 0, so its 
intersection with the above cone is not piecewise linear. The Okounkov body
$\Delta_{Y_{\bullet}}(Z;D)$ is therefore non-polyhedral.
\end{proof}


\begin{thebibliography}{RWY}
\bibitem{LB01} L. B\u adescu, \emph{Algebraic surfaces}, Universitext, Springer-Verlag, New York, 2001.

\bibitem{BKS04}  T. Bauer , A. K\"uronya, T. Szemberg, \emph{Zariski chambers, volumes and stable base loci}, Journal f\"ur die reine und angewandte Mathematik \textbf{576} (2004), 209--233. 

\bibitem{BCHM} C.~Birkar, P.~Cascini, C.~Hacon, J.~McKernan, \emph{Existence of minimal models for varieties of log general type}, \texttt{arXiv:math.AG/0610203}.

\bibitem{C} S. Dale Cutkosky,  \emph{Irrational asymptotic behaviour of Castelnuovo--Mumford regularity},  Journal f\"ur die reine und angewandte Mathematik \textbf{522} (2000), 93--103.

\bibitem{flenneretal} H.~Flenner, L.~O'Carroll, W.~Vogel, \emph{Joins and intersections}.
Springer Monographs in Mathematics. Springer-Verlag, Berlin, 1999

\bibitem{Fu82} T. Fujita, {\it Canonical rings of algebraic varieties}.   
Classification of algebraic and analytic manifolds (Katata, 1982),  65--70, 
Progr. Math., {\bf 39}, Birkh\"auser Boston, Boston, MA, 1983. 

\bibitem{Gr}
A. Grothendieck, \emph{El\'ements de g\'eom\'etrie alg\'ebrique} IV 3, Publ.Math. IHES,
\textbf{28}, (1966).

\bibitem{HS04}
M. Haiman and B. Sturmfels, \emph{Multigraded Hilbert schemes},
J. Alg. Geom. \textbf{13}, no. 4, 725-769, 2004.

\bibitem{Ha77}
R. Hartshorne, \emph{Algebraic Geometry}, 
Graduate Texts in Mathematics, vol. \textbf{52}, Springer-Verlag, New York, 1977.


\bibitem{KK08}  K. Kaveh and A. Khovanskii, 
\emph{Convex bodies and algebraic equations on affine varieties},  \texttt{arXiv:0804.4095}. 

\bibitem{KLM09} A. K\" uronya, V. Lozovanu, C. Maclean, \emph{Volume functions of linear series}, \texttt{arXiv:1008.3986}.

\bibitem{La04}
R. Lazarsfeld, \emph{Positivity in Algebraic Geometry I-II},
Ergebnisse der Mathematik und ihrer Grenzgebiete. 3. Folge, vol. \textbf{48-49}, Springer-Verlag, Berlin
Heidelberg, 2004.

\bibitem{LM08}
R. Lazarsfeld and M. Musta\c t\u a, \emph{Convex bodies associated to linear series},
Ann. Scient. \' Ec. Norm. Sup., {\bf 4} s\' erie, t. {\bf 42}, (2009), pp 783-835.

\bibitem{Morrison} D.~R.~Morrison, \emph{On K3 surfaces with large Picard number}, Invent. Math. \textbf{75} (1984), 105--121.

\bibitem{Ok96} A. Okounkov, \emph{Brunn-Minkowski inequalities for multiplicities}, Invent. Math {\bf 125} 
(1996) pp 405-411. 

\bibitem{Ok03} A. Okounkov, \emph{Why would multiplicities be 
log-concave?} in The orbit method in geometry and physics, Progr. Math. {\bf 213}, 2003 pp 329-347. 

\bibitem{Sh94}
I. R. Shafarevich, \emph{Basic Algebraic Geometry: Varieties in Projective Space}, Springer, Berlin Heidelberg New York, 1994.

\bibitem{Z62} O. Zariski, \emph{The theorem of Riemann--Roch for high multiples of an effective 
divisor on an algebraic surface}. Annals of Mathematics (2) {\bf 56} (1962), 560--615. 
\end{thebibliography}
\end{document}